\documentclass[12pt]{article}
\usepackage[intlimits]{amsmath}
\usepackage{amsmath,amsxtra,amssymb,latexsym, amscd,amsthm,pb-diagram,fancyhdr}
\usepackage[mathscr]{eucal}

\usepackage{indentfirst} 
\def\leq{\leqslant}
\def\geq{\geqslant}

\usepackage{amssymb}
\usepackage{array}
\usepackage{hhline}
\usepackage{longtable}
\usepackage{tabularx}

\usepackage[pctex32]{graphics}
\usepackage{graphicx}
\usepackage{epstopdf}
\usepackage{amsmath}
\usepackage[mathscr]{eucal}
\usepackage{amsfonts}
\usepackage{amsthm}
\usepackage{longtable}

\newtheorem{thm}{Theorem}[section]
\newtheorem*{thm*}{Theorem}
\newtheorem{cor}[thm]{Corollary}
\newtheorem{lem}[thm]{Lemma}

\theoremstyle{definition}

\theoremstyle{remark}

\def\G{{\mathscr G}}

\def\M{{\mathcal M}}

\def\Lb{{\mathbb L}}

\def\Nb{{\mathbb N}}

\def\Pb{{\mathbb P}}
\def\Qb{{\mathbb Q}}
\def\Rb{{\mathbb R}}

\def\A{{\mathcal A}}
\def\B {{\mathcal B}} 
\def\C {{\mathcal C}}

\def\E {{\mathcal E}}
\def\F {{\mathcal F}}
\def\G {{\mathcal G}}

\def\L {{\mathcal L}}
\def\M {{\mathcal M}}

\def\P {{\mathcal P}}

\def\be{\beta}

\def\Om{\Omega}

\def\ep {\varepsilon}
\def\phi{\varphi}
\def\si{\sigma}

\def\bar{\overline}

\def\Bo{\Box}
\def\tri{\triangle}

\renewcommand{\proofname}{Proof}

\def\bar{\overline}

\def\leq{\leqslant}
\def\geq{\geqslant}

\usepackage{ifpdf}
\usepackage{grffile}
\usepackage{ifpdf}
\ifpdf
  \usepackage{epstopdf}
\fi
\usepackage{wrapfig}
\begin{document}

\title{The existence of a measure-preserving bijection from a unit square to a unit segment }         
\author {Cong-Dan Pham\\Duy Tan University\\congdan.pham@gmail.com}
\maketitle
\begin{abstract}
In this paper, we prove the existence of a measure-preserving bijection from unit square to unit segment. This bijection is also called the probability isomorphism between two probability spaces. Then we give a new proof of the existence of the independent random variables on Borel probability space $([0,1],\B([0,1]),\Lb)$ that their distribution functions are the given distribution functions.
\end{abstract}
\section{Introduction}
In classical analysis theory, there are strange functions from a unit segment to a unit square. A bijection is called a function of type Cantor and a continuous surjection is called a function of type Peano (see \cite{Kha05}). In this paper we study the question of the existence of a bijection between two the Borel probability spaces $([0,1],\B([0,1]),\Lb)$ and $([0,1]^2,\B([0,1]^2),\Lb)$ (maybe difference a set with null measure), where $\Lb$ is the Lebesgue measure. The difference a set with null measure means that there exist two set $B\subset [0,1]^2, K\subset [0,1]$ with $\Lb(B)=\Lb(K)=1$ and the bijection maps from $B$ to $K$. This bijection satisfies that it and it's reverse function are measure-preserving functions. We also call this type bijection being the bi-measure-preserving function or probability isomorphism. 
Our main result reads then as follows:
\begin{thm}\label{MTH}
There exists a probability isomorphism $f$ between two Borel probability spaces the unit square and the unit segment.
\end{thm}
This result is applied to give a new proof about the existence of the independent random variables on Borel probability space $([0,1],\B([0,1]),\Lb)$ that their distribution functions are given distribution functions.
In classical probability theory, we know that, for a finite number of the given distribution functions, there exists a family of independent random variables on $([0,1],\B([0,1]),\Lb)$ such that theirs distributions are the given distributions. This result is proved by using Rademacher functions (see for instance \cite{No},\cite{bil08}).

\section{Proof of Theorem 1.1}
To prove Theorem, we need somes lemmas as follows:
\begin{lem}\label{l1}
Let $\phi:\Om\to(E,\M)$ where $\Om, E$ are two space, $\M$ is a $\si-$algebra generated by a collection of sets $\C$. Set $f^{-1}(\M):=\{f^{-1}(B): B\in\M\}$. Then $f^{-1}(\M)$ is a $\si-$algebra generated by $f^{-1}(\C).$
\end{lem}
One proved this lemma in somewhere in measure theory.
Now, we present an important theorem to apply in the next part:
\begin{thm}[probability isomorphism]\label{MB}
Let two probability spaces $(\Om,\F,\Pb), (E,\A,\Qb)$ and a bijection $\phi:\Om\to\E$ such that:

 $\F,\A$ are respectively generated by two collection of sets $\E,\M$ i.e. $\F=\si(\E),\A=\si(\M)$ and 
$f(\E)=\{f(A),\forall A\in\E\}\subset \A$; $f^{-1}(\M)=\{f^{-1}(B),\forall B\in\M\}\subset\F$ and $\Pb(f^{-1}(B))=\Qb(B)\forall B\in\M$. 

Then $f,f^{-1}$ are measurable ($f$ is bi-measurable) and if $\M$ is closed by intersection so $f,f^{-1}$ are measure-preserving ($f$ is called bi-measure-preserving or probability isomorphism.)
\end{thm}

\begin{proof}
By Lemma \ref{l1} then $f^{-1}(\A)$ is $\si-$algebra generated by $f^{-1}(\M)$, moreover $f^{-1}(\M)\subset \F$, this implies $f^{-1}(\A)\subset \F.$ It is similar to have that $f(\E)\subset \A.$ Therefore
$f,f^{-1}$ are measurable and there exist a bijection $\phi:\F\to\A$ such that $\phi(A):=f(A).$ 
Now, consider $\L=\{B\in\A:\Pb(f^{-1}(B))=\Qb(B)\}$. By assumption, $\M\subset \L.$ We prove that $\L$ is a $\si-$additive class.
In fact, $f^{-1}(\emptyset)=\emptyset$ and $f^{-1}(E)=\Om$ then $\emptyset,E\in\L.$
Consider $A\subset B$ and $A,B\in\L$, we have $\Pb(f^{-1}(B\setminus A))=\Pb[f^{-1}(B)\setminus f^{-1}(A)]=\Pb(f^{-1}(B))-\Pb(f^{-1}(A))=\Qb(B)-\Qb(A)$. Therefore $B\setminus A\in\L.$
It remains to prove the $\si-$additivity. Let $\{B_n\}_{n\geq 1}$ are pairwise disjoint. So we get 
\begin{align*}
\Pb[f^{-1}(\cup_{n\geq 1} B_n)]=\Pb[\cup_{n\geq 1}f^{-1}(B_n)]=\sum_{n\geq 1}\Pb[f^{-1}( B_n)]
=\sum_{n\geq 1}\Qb( B_n)=\Qb(\cup_{n\geq 1} B_n)
\end{align*}
This equation implies that $\cup_{n\geq 1} B_n\in\L.$ Therefore, $\L$ is $\si-$additive class concluding $\M.$ On the other hand, $\M$ is closed by intersection so $\si-$algebra $\A$ generated by $\M$ is also $\si-$additive class generated by $\M.$ Since, $A\subset \L.$ By the definition of $\L$, $\L\subset\A$ then $\L=\A$. It means that $\Pb(f^{-1}(B))=\Qb(B)\forall B\in\A$ and $f,f^{-1}$ are measure-preserving. 
\end{proof}

Now, we return to the proof of Theorem \ref{MTH}. Denote $\Bo=[0,1]^2, \tri=[0,1].$ We will show that there exist a Borel set $B$ on $\Bo$ such that $\Lb(B)=1$ and a bijection $f:B\to\tri$ such that $f,f^{-1}$ are Borel measurable, measure-preserving. We know that any Borel set of $\Rb^2,\Rb$ has a positive measure then it has continumn cardinality. We will construct the bijection $f$ and the set $B$ as follows:
\begin{wrapfigure}{R}[0pt]{0.3\textwidth}
 \vspace{-20pt}
  \begin{center}
  \includegraphics[width=0.3\textwidth]{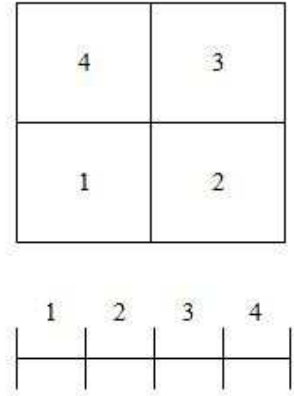}
   \vspace{-20pt}
    \caption{Construction of bijections}
    \end{center}
  \vspace{-20pt}
\end{wrapfigure}

Step 1: The square is divided into $4$ smaller squares that are numerated as figure above. The unit segment is divided uniformly into $4$ segments and they also numerated.

Next step we redivide uniformly every smaller squares and small segments into $4$ parts. We continue this process until infinite steps and numerate the squares and segments by $1,2,3,4$ at every step.
Consider $4$ apexes of the square by coordinates: $(0,0);(0,1);(1,0);(1,1).$ Set $M:=\{(x,y)\in\Bo: x \text{ or }y=\frac{a}{2^n}, 0\leq a\leq 2^n,a\in\Nb\}$ and $\Bo^*=\Bo\setminus M.$ $M$ is a Borel set and has the Lebesgue measure $\Lb(M)=0.$
Similarly, we put a Descartes system to have two endpoints of the unit segment $(0,0);(1,0).$ Set $C=\{x\in\tri: x=\frac{a}{4^n},a\in\Nb,n\in\Nb^*\}.$ $C$ is denumerable and $\Lb(C)=0.$ Denote $\Bo_n^{j_1j_2...j_n},\tri_n^{j_1j_2...j_n}$ be respectively the closed square, close segment obtained at the $n-th$ step and $j_k$ is order number of the square, segment at $k-th$ step contains it ($j_k\in\bar{1,4}$). Set ${}^*\Bo_n^{j_1j_2...j_n}:=\Bo_n^{j_1j_2...j_n}\setminus M,{}^*\tri_n^{j_1j_2...j_n}:=\tri_n^{j_1j_2...j_n}\setminus C.$

We construct a sequence of bijections $\{f_n\}_{n\geq 1}$, $f_n:\Bo^*\to\tri^*$, satisfying $f({}^*\Bo_n^{j_1j_2...j_n})={}^*\tri_n^{j_1j_2...j_n}$, we can chose a bijection like that because two sets have the same continumn cardinality. With this construction, we see that for all $m>n$:
$|f_m(x)-f_n(x)|<\frac{1}{4^n}$. Indeed, let $x\in{}^*\Bo_n^{j_1j_2...j_n...j_m}\subset{}^*\Bo_n^{j_1j_2...j_n}$ so $f_n(x)\in{}^*\tri_n^{j_1j_2...j_n}$ and $f_m(x)\in{}^*\tri_n^{j_1j_2...j_n...j_m}\subset{}^*\tri_n^{j_1j_2...j_n}$. Therefore $|f_m(x)-f_n(x)|<|{}^*\tri_n^{j_1j_2...j_n}|=\frac{1}{4^n}.$ This implies that $\{f_n\}$ converges uniformly to $f$ and $f:\Bo^*\to\tri.$ Let us prove that $f$ is injective and $\tri^*\subset f(\Bo^*).$ We prove by contradiction. Suppose that there exist $x\ne y$ such that $f(x)=f(y).$ Because $x\ne y$ then there exist $N$ such that $x\in{}^*\Bo_n^{j_1j_2...j_{N-1}j_N};y\in{}^*\Bo_n^{j_1j_2...j_{N-1}j'_N}$ with $j_N\ne j'_N.$ If $|j_N-j'_N|>1$ then $\tri_n^{j_1j_2...j_N};\tri_n^{j_1j_2...j'_N}$ are two disjoint segments for all $n\geq N$. Moreover, $f_n(x)\in\tri_n^{j_1j_2...j_N}, f_n(y)\in\tri_n^{j_1j_2...j'_N}$, take $n$ tend to infinity we get $f(x)\in\tri_n^{j_1j_2...j_N}, f(y)\in\tri_n^{j_1j_2...j'_N}$, it implies $f(x)\ne f(y).$ This is contradictory with the assumption $f(x)=f(y)$ so $|j_N-j'_N|=1.$ Hence we have three cases $(j_N,j'_N)\in\{(1,2);(2,3);(3,4)\}$ (suppose $j_N<j'_N$). Since $\lim f_n(x)=\lim f_n(y)=f(x)=f(y)$; $f_n(x),f_n(y)$ always belong to two segments side by side at the step $n$ and  $f_n(x)$ always belong to segments numerated by $4$ and $f_n(y)$ always belong to segments numerated by $1$. It means that $f_n(x)\in\tri_n^{j_1j_2...j_N44...4};f_n(y)\in\tri_n^{j_1j_2...j'_N11...1}.$ By the definition of $f_n$, it implies
$x\in{}^*\Bo_n^{j_1j_2...j_N44...4};y\in{}^*\Bo_n^{j_1j_2...j_N11...1}.$ Hence $x\in\bigcap_{n\geq N}\Bo_n^{j_1j_2...j_N44...4};y\in\bigcap_{n\geq N}\Bo_n^{j_1j_2...j_N11...1}.$ 

\begin{wrapfigure}{R}[0pt]{0.3\textwidth}
 \vspace{-20pt}
  \begin{center}
  \includegraphics[width=0.3\textwidth]{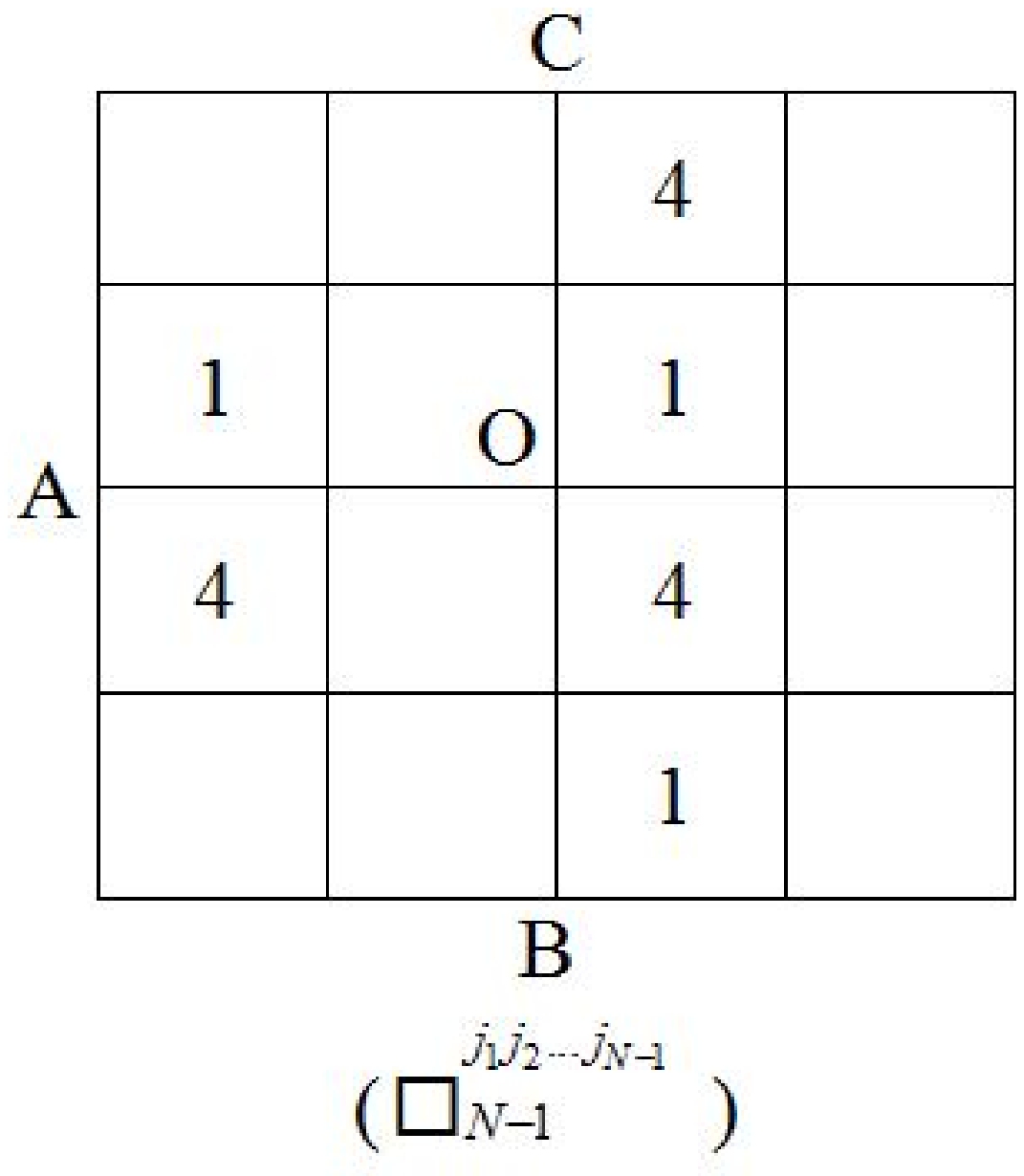}
   \vspace{-20pt}
    
    \end{center}
  \vspace{-30pt}
\end{wrapfigure} 

Consider three case:
\begin{align*}
 j_N=1,j'_N=2 \text{ then } x=A,y=B\\
 j_N=2,j'_N=3 \text{ then } x=O,y=O\\
 j_N=3,j'_N=4 \text{ then } x=C,y=A
\end{align*}
All three cases are contradictory because $x\be y$ and $x,y\notin M$. Therefore, the assumption of the proof by contradiction is wrong then $f$ is injective.
Now, we prove $\tri^*\subset f(\Bo^*).$ Indeed, let $y\in\tri^*$, at the step $n$ suppose that $y\in{}^*\tri_n^{j_1j_2...j_n}$. This implies $f_n^{-1}(y)\in{}^*\Bo_n^{j_1j_2...j_n}.$ The notation $"0"$ indicate the interior of a set. Because ${}^*\Bo_n\subset {}^0\Bo_n\subset\Bo_n$ then there exist a sequence $\{n_k\}$ such that $f^{-1}_{n_k}(y)\in{}^*\Bo_{n_k}^{j_1j_2...j_{n_k}}$ and $\Bo_{n_k}^{j_1j_2...j_{n_k}j_{n_{k+1}}}\subset {}^*\Bo_{n_k}^{j_1j_2...j_{n_k}}.$  Hence,
$$\bigcap_{k=1}^{\infty}{}^*\Bo_{n_k}^{j_1j_2...j_{n_k}}=\bigcap_{k=1}^{\infty}{}^0\Bo_{n_k}^{j_1j_2...j_{n_k}}=\bigcap_{k=1}^{\infty}\Bo_{n_k}^{j_1j_2...j_{n_k}}=\{x\}.$$ 
The last equality is followed by the sequence of closed squares has the lengths of the sides tend to $0$ so the intersection is only one point. 
Let us prove $f(x)=y.$ Because $x\in{}^*\Bo_k^{j_1j_2...j_k}$ then $f_{n_k}(x)\in{}^*\tri_k^{j_1j_2...j_k}.$ On the other hand, $y\in{}^*\tri_k^{j_1j_2...j_k}$, this implies 

$$
|f_{n_k}(x)-y|\leq|{}^*\tri_k^{j_1j_2...j_k}|=\frac{1}{4^{n_k}}\to 0\text{ when }k\to\infty. 
$$
Therefore, $f_{n_k}(x)\to y$, we also have $f_n(x)\to f(x)$ so $y=f(x).$

Step 2: The set $\tri\setminus f({}^*\Bo)\subset\tri\setminus {}^*\tri=C$ so it is denumerable or finite. We extend $f$ such that it maps bijectively from a subset $M'$ of $M$ to $\tri\setminus f({}^*\Bo).$  Set $B={}^*\Bo\cup M'$ then $\Lb(B)=1$ and $f$ is a bijection from $B$ to $\tri.$
It remains us to prove $f$ is measurable and measure-preserving. The family $\{\emptyset,{}^*\Bo_{n}^{j_1j_2...j_{n}},\P(M')$ is the generating collection of the $\si-$algebra Borel $\B(B)$ where $\P(M')$ is the family of all subsets of $M'$. The family $\{\emptyset,{}^*\tri_{n}^{j_1j_2...j_{n}},\P(C)\}$ is the generating collection of $\B(\tri).$ Remark that the two generating collections are closed by the intersection.
The image of ${}^*\Bo_{n}^{j_1j_2...j_{n}}$ by $f$ is ${}^*\tri_{n}^{j_1j_2...j_{n}}$ and maybe add a denumerable or finite number of points of $C.$ So $f({}^*\Bo_{n}^{j_1j_2...j_{n}})\in\B(\tri).$
$$
\Lb(f({}^*\Bo_{n}^{j_1j_2...j_{n}}))=\Lb({}^*\tri_{n}^{j_1j_2...j_{n}})=\frac{1}{4^n}=\Lb({}^*\Bo_{n}^{j_1j_2...j_{n}}).
$$
The image of ${}^*\tri_{n}^{j_1j_2...j_{n}}$ by $f^{-1}$ is ${}^*\Bo_{n}^{j_1j_2...j_{n}}$ and maybe subtract a denumerable or finite number of points of $f^{-1}(C)$ then $f^{-1}({}^*\tri_{n}^{j_1j_2...j_{n}})\in\B(B).$ Since $f$ satisfies Theorem \ref{MB} then $f,f^{-1}$ are bijection, measurable and measure-preserving i.e. $f$ is a probability isomorphism.  

\hfill \textbf{$\Bo$}

There are some consequences as follows:
\begin{cor}
There exist a bijection from the unit square to the unit segment such that it is a probability isomorphism between two Lebesgue spaces (Lebesgue $\si-$algebra). 
\end{cor}
\begin{proof}
We proved the existence of probability isomorphism from $\Bo$ to $\tri$ with the difference null set i.e. the bijection from $B$ to $\tri$. It is still open for the question of the existence a bijection from $\Bo$ to $\tri$, Borel measurable and is measure preserving. Now we will prove that there is also a probability isomorphism from $\Bo$ to $\tri$ which is Lebesgue measurable. Indeed, let a set $Y\subset\tri$ such that the cardinality of $Y$ is continumn and $\Lb(Y)=0$. Set $X:=f^{-1}(Y).$ Take a arbitrary bijection from $(\Bo\setminus B)\cup X$ to $Y$ (a bijection between two continumn sets), combine with the bijection
$f_{|B\setminus X}:B\setminus X\to\tri\setminus Y$ we get the bijection $F:\Bo\to\tri.$ For simplicity, we still denote $F$ by $f$.
We proved $f$ is Borel measurable and Borel measure-preserving . Now we prove for all Lebesgue sets.
Firstly, let us prove $\Lb(X)=0$. For all $\ep>0$, there exist an open set $G$ such that $Y\subset \G$ and $\Lb(G)<\ep.$ Since, $X\subset f^{-1}(G)$ and $\Lb(f^{-1}(G))=\Lb(G)<\ep.$ This implies $X$ is Lebesgue measurable and $\Lb(X)=0.$ For every Lebesgue measurable set $H$ of $\tri$ we can write $H=H_1\cup H_2$ where $H_1$ is Borel set, $\Lb(H_2)=0$ and they are disjoint. So $f^{-1}(H)=f^{-1}(H_1)\cup f^{-1}(H_2)$ is Lebesgue measurable, $\Lb(f^{-1}(H))=\Lb(f^{-1}(H_1))=\Lb(H_1)=\Lb(H).$
So $f$ is a probability isomorphism between two Lebesgue probability spaces $\Bo$ and $\tri.$
\end{proof}

\begin{cor}
There exists a probability isomorphism between two Lebesgue spaces $[0,1]^n$ and $[0,1]^m$, where $m,n$ are positive integers.
\end{cor}
Now, we give an application of Theorem \ref{MTH}. We recover the existence of random variables on the probability space $\tri$ associate the given distribution functions.
\begin{cor}
Let $F_1,F_2,...,F_n$ be the distribution functions and $(\tri,\B(\tri),\Lb)$ be the Borel probability space. Then there exits $n$ independent random variables $X_1,...,X_n:\tri\to\Rb$ such that their distribution functions are $F_{X_1}=F_1,...,F_{X_n}=F_n.$
\end{cor}
\begin{proof}
This result is proved by using Ramdemacher functions,( see \cite{No},\cite{bil08}). 
Let us explain how do we prove it. Let a probability isomorphism $f:\tri\to B^n.$
Let $\bar{X}_i: B^n\to\Rb$ define by $\bar{X}_i(x_1,..,x_i,...,x_n)=\sup\{t\in\Rb:F_i(t)<x_i\}$. It is clear that $\{\bar{X}_i\}_{i=\bar{1,n}}$ are independents and one proved that they have respectively the distribution functions $F_1,..,F_n.$ Set $X_i=\bar{X}_i\circ f$, then $\{X_i\}$ are independents and their distribution functions:
$$
F_{X_i}(x)=\Lb(X_i^{-1}(-\infty,x])=\Lb(f^{-1}\circ\bar{X}_i^{-1}(-\infty,x])=\Lb(\bar{X}_i^{-1}(-\infty,x])=F_i(x).
$$
\section{Open question}
It is still open for the existence of a probability isomorphism between two Borel probability spaces $([0,1],\B([0,1]),\Lb)$ and $([0,1]^2,\B([0,1]^2),\Lb)$ such that it has not the difference of a null measure set.
\end{proof}
\section*{Acknowledgments} I would like to thank my old professors of Faculty of mathematics, Hanoi university of education for suggesting this problem.

\newpage
\bibliographystyle{plain}
\bibliography{dl}
\thispagestyle{empty}
\end{document}